\documentclass{amsart}

\usepackage{amsmath,amssymb}
\usepackage{tikz,graphicx}
\usetikzlibrary{shapes.geometric}

\usepackage[ruled,vlined]{algorithm2e}
\SetKwProg{Function}{Function}{}{}
\SetKwInOut{Input}{Input}
\SetKwInOut{Output}{Output}

\usepackage{hyperref}

\newcommand{\Z}{\mathbb{Z}}
\newcommand{\Q}{\mathbb{Q}}
\newcommand{\K}{\mathbf{K}}
\newcommand{\OK}{\mathcal{O}_K}
\newcommand{\Norm}{\mathrm{Norm}}

\newcommand{\tarc}{\mbox{\large$\frown$}}
\newcommand{\arc}[1]{\stackrel{\tarc}{#1}}

\newtheorem{lemma}{Lemma}[section]

\newtheorem{proposition}[lemma]{Proposition}

\theoremstyle{definition}

\theoremstyle{remark}

\begin{document}

\title{Division algorithms for norm-Euclidean imaginary quadratic fields}
\author{Fran\c{c}ois MORAIN}
\thanks{LIX, CNRS, INRIA, \'Ecole Polytechnique, Institut
Polytechnique de Paris, Palaiseau, France}
\email{morain@lix.polytechnique.fr}
\date{\today}

\maketitle

\begin{abstract}
The list of norm-Euclidean imaginary quadratic fields
is known and finite. For each known case, we give a division algorithm
that finds a remainder at distance less than the Euclidean minimum of
the field.
\end{abstract}

\section{Introduction}

Let $\K$ be a number field, $\OK$ its ring of integers and $\Norm(\xi)$ the
norm of an element $\xi\in \K$.
Not all $\OK$ are norm-Euclidean. Depending on the signature of $\K$,
there may be a finite or (conjectured) infinite number of such
fields. A necessary reading is \cite{Lemmermeyer1995}.
The {\em Euclidean minimum} of $\xi \in \K$ is given by
$$m_\K(\xi) = \mathrm{inf}\{|\Norm(\xi - \gamma)|, \gamma \in \OK\}.$$
The {\em Euclidean minimum} of $\K$ is
$$M(\K) = \mathrm{sup}\{m_\K(\xi), \xi \in \K\}.$$
See \cite{Cerri2007,Lezowski2014} for efficient algorithms
to compute the euclidean minimum for elements (resp. fields) of rather
large degree. 

Many proofs in the literature are targeted to the proof that
$M(\K) < 1$ and sometimes exhibit an algorithm that, given $\xi \in \K$,
finds a $\gamma$ such that $|\Norm(\xi - \gamma)| < 1$. Such an
algorithm we name $1$-division. It is even rarer to find
$M(\K)$-division algorithms, where the condition is now $|\Norm(\xi -
\gamma)| \leq M(\K)$ (see~\cite{Kaiblinger2011} for some results in
that direction). Such algorithms are interesting {\it per se},
and can be used for computing the gcd of two integers of $\K$, though
alternative techniques exist
(see \cite{KaRo1989,AgFr2004,DaFr2005,AgFr2006} and the references
therein). It can be used in the explicit computation of higher
reciprocity laws see \cite{DaFr2005} and the references within
(also~\cite{CaSc2010,JoLaNgNa2021}). Some cryptographic applications
exist, see e.g. \cite{KiLe2017}.

Using centered remainders, we get $M(\Q)=1/2$.
We start our work on $M(\K)$-division algorithms with the case of
imaginary quadratic fields, which is easier than the real one (see
\cite{Morain2026a,Morain2026b,Morain2026c}).

The complete list of norm-Euclidean imaginary quadratic fields is
known for quite a long time.
The proof of \cite[Theorem 246]{HaWr85} yields a 1-division
algorithm for all the cases of $\Q(\sqrt{m})$ that are euclidean,
i.e. $m \in \{-1, -2, -3, -7, -11\}$. When $m \in \{-1,
-2\}$, the algorithms are in fact $M(\K)$-division algorithms (with
respective values of $M(\K)$ being $1/2$ and $3/4$).
We are left with the case where $m \in \{-3, -7, -11\}$. Note that the
case $m=-3$ was already treated in~\cite{Meissner1909}, whereas some
suboptimal versions appeared later (see \cite{Williams1985,ScWi1995}).

\section{Proofs}

Elements of $\OK$ are of the form $u_0 + u_1 \omega$ where $\omega =
(1+\sqrt{-m})/2$ and $u_0$, $u_1$ rational integers. The conjugate of
$U = u_0 + u_1 \omega$ is
$U' = u_0 + u_1 \omega'
= (u_0 + u_1) - u_1 \omega$ and remember that $\Norm(U) = U\cdot U'$.
Put $\ell = (1-m)/4 > 0$; the norm function is $\Norm(x+y\omega) =
f_m(x, y) = x^2 + x y + \ell y^2$ and is clearly positive. We give $M
= M(\K)$ and a point reaching it
in Table~\ref{mMP}.
\begin{table}[hbt]
$$\begin{array}{|c|c|c|c|}\hline
m & -3 & -7 & -11 \\ \hline
M(\K), P & 1/3, (1/3, 1/3) & 4/7, (2/7, 3/7) & 9/11, (3/11, 5/11)
\\
\hline
\end{array}$$
\caption{Table of triples $(m, M, P)$.\label{mMP}}
\end{table}

Suppose we want to compute the Euclidean division of $U = u_0 + u_1
\omega$ by $V = v_0 + v_1 \omega \neq 0$. We start from
$$\frac{U}{V} = \frac{U V'}{\Norm(V)} =
\frac{w_0 + w_1 \omega}{N} =
\left\lfloor \frac{w_0}{N}\right\rceil +
\left\lfloor \frac{w_1}{N}\right\rceil \omega
+ (a + b \omega)
= (q_0 + q_1 \omega) + (a + b \omega)$$
with $q_0$ and $q_1$ rational integers (and forming the quotient $Q$), $a,
b \in \mathcal{S} = [-1/2, 1/2] \times [-1/2, 1/2]$. To finish the
computation, we need to find two rational integers $(\delta_a,
\delta_b)$ such that
$$\Norm(a+\delta_a, b+\delta_b) \leq M$$
so that
$$\frac{U}{V} = (q_0-\delta_a) + (q_1-\delta_b) \omega + r_0 + r_1
\omega$$
and $\Norm(r_0 + r_1\omega) \leq M$, leading to
$$U = Q V + R$$
and $\Norm(R) \leq M \cdot \Norm(V) < \Norm(V)$.

The three cases can be treated the same way, and we assume we fix $(m,
M)$ and set $f = f_m$. From the figures below,
it seems that we need 5 ellipses to cover the square $\mathcal{S}$.
Let us introduce the equations of these ellipses:
\begin{eqnarray}
E_{0, 0}: f(x, y) &\leq & M,\\
E_{0, 1}: f_{0, 1}(x, y) := x^2 + x (y-1) + \ell (y-1)^2 &\leq & M,\\
E_{1, 0}: f_{1, 0}(x, y) := (x-1)^2 + (x-1) y + \ell y^2 &\leq & M,\\
E_{0, -1}: x^2 + x (y+1) + \ell (y+1)^2 &\leq & M,\\
E_{-1, 0}: (x+1)^2 + (x+1) y + \ell y^2 &\leq & M.
\end{eqnarray}
Note that $E_{0, 1}$ and $E_{0, -1}$ are symmetrical w.r.t. $(0, 0)$
using $(x, y) \leftrightarrow (-x, -y)$. Ditto for $E_{1, 0}$ and
$E_{-1, 0}$. We remark that the square $[-1/2, 0] \times [0, 1/2]$ and
its symmetrical can be covered by $E_{0, 0}$ only.
\begin{lemma}
If $a$ and $b$ have opposite signs, then $(a, b) \in {E}_{0,
0}$ that is $f(a, b) \leq M$.
\end{lemma}

\medskip
\noindent
{\em Proof:}
We compute the maximal values of $f$ on the square:
$$f(-1/2, 0) = \frac{1}{4}, f(0, 1/2) = \frac{\ell}{4},$$
and these values are less than the possible $M$'s.
We terminate using the convexity of $E_{0, 0}$. $\Box$

\medskip
Proofs are required for covering $\mathcal{S}_0 = [0, 1/2]\times [0,
1/2]$, using the ellipses $E_{0, 0}$, $E_{0, 1}$ and
$E_{1, 0}$. The covering of $[-1/2, 0] \times [-1/2, 0]$ will follow
by symmetry.

\medskip
Let us introduce some points:
$P = (x_P, 1/2)$ is the intersection of
$\mathcal{S}_0$ with $E_{0, 0}$, $Q = (1/2, 1/2)$ and $R = (1/2, y_R)$
is another intersection point of $\mathcal{S}_0$ with $E_{0, 0}$; $S =
(1/2, y_S)$ is the intersection point of $E_{0, 1}$. All these numbers
are exact and signs of algebraic expressions easy to deal with
(see~\cite{Morain2026a} for more explanations).

\begin{lemma}
Fix $\ell \in \{1, 2, 3\}$.
There exists a (unique) point of intersection $I$ to the three ellipses
$E_{0, 0}$, $E_{1, 0}$, $E_{0, 1}$ whose coordinates are
$$I = \left(\frac{\ell}{4 \ell-1}, \frac{2\ell-1}{4\ell-1}\right).$$
\end{lemma}

\medskip
\noindent
{\em Proof:}
We rewrite all quantities as functions of $\ell$:
$$m = 1-4 \ell, M = {\frac {{\ell}^{2}}{4\,\ell-1}}$$
and we factor the resultant of $f$ and $f_{01}$ w.r.t. $y$, which
leads to the result. $\Box$

\begin{proposition}
The ellipses $E_{0, 0}$, $E_{0, 1}$ and $E_{1, 0}$ cover $\mathcal{S}_0$.
\end{proposition}

\medskip
\noindent
{\em Proof:}
1) First, suppose $m=-3$ and therefore $M=1/3$.
There is a symmetry $(x, y) \leftrightarrow (y, x)$
which swaps $E_{0, 1}$ and $E_{1, 0}$, so that we may assume that $y
\geq x$. Suppose that there exists $(x, y)$ such that $f(x, y) > 1/3$
and $f_{0, 1}(x, y) > 1/3$. We deduce
$$3 y^2 \geq x^2 + x y + y^2 > 1/3, \quad
3 y^2 - 3 y + 1 > 1/3$$
that is $y > 1/3$ and $y^2-y+2/9 = (y-1/3) (y-2/3) > 0$, leading to a
contradiction. Hence $(x, y)$ belongs to one of the two ellipses.

\begin{figure}[hbt]
\begin{tikzpicture}[>=latex,scale=2.25]
\draw[->] (0, -1.8) -- (0, 1.8);
\draw[->] (-2, 0) -- (2, 0);
\draw[-] (-0.5, -0.5) rectangle (0.5, 0.5);
\draw[dashed] (-1, 1) -- (1, -1); \node at (0.9, -0.7) {$\mathcal{L}$};
\draw[dashed] (-1, -1) -- (1, 1);
\draw (0, 0) circle [x radius=0.471, y radius=0.816, rotate=45];
 \node (E00) at (-0.3, 0.15) {$E_{0, 0}$};
\draw (0, 1) circle [x radius=0.471, y radius=0.816, rotate=45];
 \node (E01) at (-0.4, 1.4) {$E_{0, 1}$};
\draw (0, -1) circle [x radius=0.471, y radius=0.816, rotate=45];
 \node (E0m1) at (0.4, -1.4) {$E_{0, -1}$};
\draw (-1, 0) circle [x radius=0.471, y radius=0.816, rotate=45];
 \node (Em10) at (-1.4, 0.4) {$E_{-1, 0}$};
\draw (1, 0) circle [x radius=0.471, y radius=0.816, rotate=45];
 \node (E10) at (1.4, -0.4) {$E_{1, 0}$};
 \node (0) at (0, 0) {$\bullet$};
 \node (am) at (-1, 0) {$\bullet$};
 \node (ap) at (1, 0) {$\bullet$};
 \node (bm) at (0, -1) {$\bullet$};
 \node (bp) at (0, 1) {$\bullet$};
 \node (P) at (0.132, 0.5) {$\circ$}; \node at (0.132, 0.7) {$P$};
 \node (Q) at (0.5, 0.5) {$\circ$}; \node at (0.5, 0.75) {$Q$};
 \node (R) at (0.5, 0.132) {$\circ$}; \node at (0.7, 0.132) {$R$};
 \node (S) at (0.5, 0.368) {$\circ$}; \node at (0.7, 0.368) {$S$};
 \node (I) at (0.333, 0.333) {$\circ$}; \node at (0.1, 0.25) {$I$};
\end{tikzpicture}
\caption{The case $m=-3$}
\end{figure}
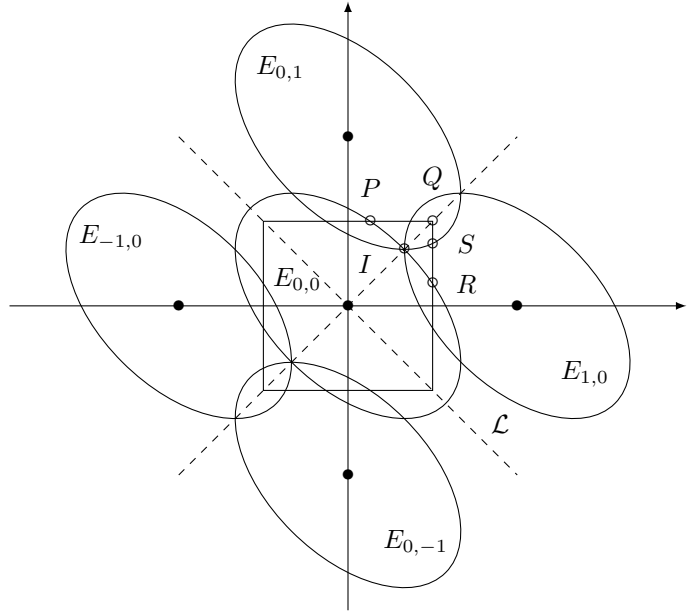

\medskip
\noindent
2) Fix a pair $(m, M) \in \{(-7, 4/7), (-11, 9/11)\}$.
From Figure~\ref{case7}, we can see that the ellipse $E_{0, 0}$ covers a
large part of $\mathcal{S}_0$, but the closed region
$PQRIP$, which we cut as $PQSI \cup SIR$ (see Figure~\ref{case7x2} and
Figure~\ref{case11x2}).
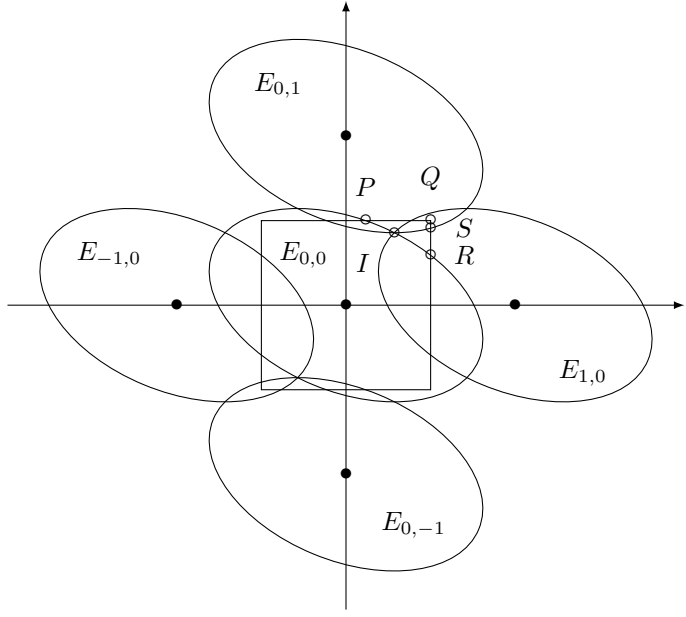
\begin{figure}[hbt]
\begin{tikzpicture}[>=latex,scale=2.25]
\draw[->] (0, -1.8) -- (0, 1.8);
\draw[->] (-2, 0) -- (2, 0);
\draw[-] (-0.5, -0.5) rectangle (0.5, 0.5);
\draw (0, 0) circle [x radius=0.849, y radius=0.509, rotate=-22.5];
 \node (E00) at (-0.25, 0.3) {$E_{0, 0}$};
\draw (0, 1) circle [x radius=0.849, y radius=0.509, rotate=-22.5];
 \node (E01) at (-0.4, 1.3) {$E_{0, 1}$};
\draw (0, -1) circle [x radius=0.849, y radius=0.509, rotate=-22.5];
 \node (E0m1) at (0.4, -1.3) {$E_{0, -1}$};
\draw (-1, 0) circle [x radius=0.849, y radius=0.509, rotate=-22.5];
 \node (Em10) at (-1.4, 0.3) {$E_{-1, 0}$};
\draw (1, 0) circle [x radius=0.849, y radius=0.509, rotate=-22.5];
 \node (E10) at (1.4, -0.4) {$E_{1, 0}$};
 \node (0) at (0, 0) {$\bullet$};
 \node (am) at (-1, 0) {$\bullet$};
 \node (ap) at (1, 0) {$\bullet$};
 \node (bm) at (0, -1) {$\bullet$};
 \node (bp) at (0, 1) {$\bullet$};
 \node (P) at (0.116, 0.5) {$\circ$}; \node at (0.116, 0.7) {$P$};
 \node (Q) at (0.5, 0.5) {$\circ$}; \node at (0.5, 0.75) {$Q$};
 \node (R) at (0.5, 0.295) {$\circ$}; \node at (0.7, 0.295) {$R$};
 \node (S) at (0.5, 0.455) {$\circ$}; \node at (0.7, 0.455) {$S$};
 \node (I) at (0.286, 0.429) {$\circ$}; \node at (0.1, 0.25) {$I$};
\end{tikzpicture}
\caption{The case $m=-7$. \label{case7}}
\end{figure}
By elementary computations, we prove that $P$, $Q$, and $S$
are all inside $E_{0, 1}$; since $I$ is on $E_{0, 1}$, this proves
that the polygone $PQSI$ is inside $E_{0, 1}$ by convexity.
In the same way, $R$ is on $E_{0, 0}$ and
inside $E_{10}$. Since the line $IR$ is below the arc $\arc{IR}$ and $IS$ is
above the arc $\arc{IS}$, this proves that the region $IRS$ is
inside $E_{1, 0}$. $\Box$

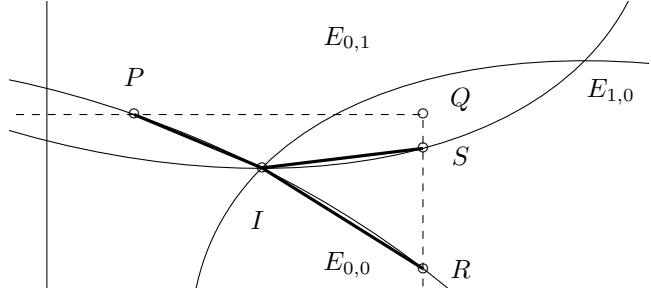
\begin{figure}[hbt]
\begin{tikzpicture}[>=latex,scale=10]
\clip (-0.05, 0.27) rectangle (0.8, 0.65);
\draw[->] (0, -1) -- (0, 2);
\draw[->] (-1, 0) -- (2, 0);
\draw[dashed] (-0.5, -0.5) rectangle (0.5, 0.5);
\draw (0, 0) circle [x radius=0.849, y radius=0.509, rotate=-22.5];
 \node (E00) at (0.4, 0.3) {$E_{0, 0}$};
\draw (0, 1) circle [x radius=0.849, y radius=0.509, rotate=-22.5];
 \node (E01) at (0.4, 0.6) {$E_{0, 1}$};
\draw (1, 0) circle [x radius=0.849, y radius=0.509, rotate=-22.5];
 \node (E10) at (0.75, 0.53) {$E_{1, 0}$};
 \node (0) at (0, 0) {$\bullet$};
 \node (ap) at (1, 0) {$\bullet$};
 \node (bp) at (0, 1) {$\bullet$};
 \node (I) at (0.286, 0.429) {$\circ$}; \node at (0.28, 0.36) {$I$};
 \node (P) at (0.116, 0.5) {$\circ$}; \node at (0.116, 0.55) {$P$};
 \node (Q) at (0.5, 0.5) {$\circ$}; \node at (0.55, 0.52) {$Q$};
 \node (R) at (0.5, 0.295) {$\circ$}; \node at (0.55, 0.295) {$R$};
 \node (S) at (0.5, 0.455) {$\circ$}; \node at (0.55, 0.445) {$S$};
  \draw[-,very thick] (0.5, 0.455) -- (0.286, 0.429);
  \draw[-,very thick] (0.116, 0.5) -- (0.286, 0.429);
  \draw[-,very thick] (0.5, 0.295) -- (0.286, 0.429);
\end{tikzpicture}
\caption{The case $m=-7$; zoom on the critical region. \label{case7x2}}
\end{figure}

\begin{figure}[hbt]
\begin{tikzpicture}[>=latex,scale=10]
\clip (-0.05, 0.32) rectangle (0.8, 0.6);
\draw[->] (0, -2) -- (0, 2);
\draw[->] (-2, 0) -- (2, 0);
\draw[dashed] (-0.5, -0.5) rectangle (0.5, 0.5);
\draw (0, 0) circle [x radius=0.963, y radius=0.512, rotate=-13.28];
 \node (E00) at (0.4, 0.34) {$E_{0, 0}$};
\draw (1, 0) circle [x radius=0.963, y radius=0.512, rotate=-13.28];
 \node (E10) at (0.75, 0.5) {$E_{1, 0}$};
\draw (0, 1) circle [x radius=0.963, y radius=0.512, rotate=-13.28];
 \node (E01) at (0.4, 0.55) {$E_{0, 1}$};
 \node (0) at (0, 0) {$\bullet$};
 \node (am) at (-1, 0) {$\bullet$};
 \node (ap) at (1, 0) {$\bullet$};
 \node (bm) at (0, -1) {$\bullet$};
 \node (bp) at (0, 1) {$\bullet$};
 \node (I) at (0.273, 0.455) {$\circ$}; \node at (0.26, 0.4) {$I$};
 \node (P) at (0.111, 0.5) {$\circ$}; \node at (0.111, 0.55) {$P$};
 \node (Q) at (0.5, 0.5) {$\circ$}; \node at (0.55, 0.51) {$Q$};
 \node (R) at (0.5, 0.360) {$\circ$}; \node at (0.55, 0.295) {$R$};
 \node (S) at (0.5, 0.474) {$\circ$}; \node at (0.55, 0.445) {$S$};
  \draw[-,very thick] (0.273, 0.455) -- (0.111, 0.5);
  \draw[-,very thick] (0.273, 0.455) -- (0.5, 0.474);
  \draw[-,very thick] (0.273, 0.455) -- (0.5, 0.360);
\end{tikzpicture}
\caption{The case $m=-11$; the critical region. \label{case11x2}}
\end{figure}
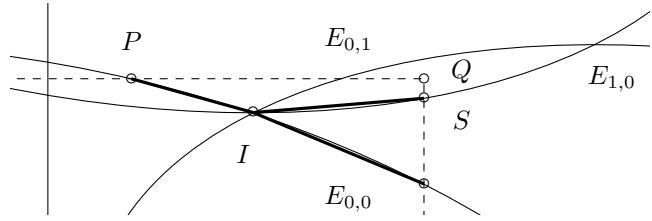

\section{Algorithms}

The approach of Meissner for $m=-3$ consists in easy inequalities on
$(a, b) \in \mathcal{S}_0$ to find a closest neighbor in $\Z[\omega]$,
among a list of four of them. In our work, we find an ellipsis
containaing $(a, b)$ by evaluating the defining equations on $(a, b)$
and waiting for one of these to be negative. We begin with $f$ since
$E_{0, 0}$ covers a large part of $\mathcal{S}_0$.

We note that:
$$f_{0, 1}(x, y) = f(x, y) - x - \ell (2y-1), \quad
f_{1, 0}(x, y) = f(x, y) - 2 x - y + 1;$$
$$f_{0, -1}(x, y) = f(x, y)+ x + \ell (2 y + 1), \quad
f_{-1, 0}(x, y) = f(x, y) + 2 x + y + 1.$$
The idea is to test the value of the various expressions with
incremental computations, to decrease the complexity of the
computations. For $(a, b) \in \mathcal{S}_0$, we first compute $F =
f(a, b) = a (a+b) + \ell b^2$. If
$F \leq M$, we are done. Otherwise, we need to check whether
$$f_{1, 0}(a, b) \leq M$$
that is
$$F-M \leq a + \ell(2 b -1).$$
We proceed with the third case if needed.
This gives Algorithm~\ref{algo2}. Some shortcuts are
possible for $\ell = 1$.

\begin{algorithm}[hbt]
\caption{$M$-euclidean division for $m \in \{-3, -7, -11\}$; condensed
version \label{algo2}}
\Function{QED($\ell$, $f$, $M$, $a$, $b$)}{
    \Input{$a, b \in \mathcal{S}$}
    \Output{a pair of integers $(\delta_a, \delta_b)$ s.t. $\Norm(a+\delta_a,
    b+\delta_b) \leq M$}
  \If{$a$ and $b$ are of opposite signs}{
  	  \Return{$(0, 0)$}\;
  }
  \tcp{$a$ and $b$ have the same sign}
  \tcp{compute $F = \Norm(a+b\omega)$}
  $F \leftarrow a (a+b) + \ell b^2$\;
  \If{$F \leq M$}{
  	 \tcp{$(a, b) \in E_{0, 0}$}
  	 \Return{$(0, 0)$}\;
  }
  $F \leftarrow F-M$\;
  $sgn \leftarrow Sign(a)$\;
  $a \leftarrow |a|$; $b \leftarrow |b|$\;
  \If{$F \leq 2 a+b-1$}{
  	 \tcp{$(a, b) \in E_{0, 1}$}
	$\sigma \leftarrow (-1, 0)$\;
  }
  \ElseIf{$F \leq a + \ell (2b-1)$}{
  	 \tcp{$(a, b) \in E_{1, 0}$}
  	 $\sigma \leftarrow (0, -1)$\;
  }
  \Else{
	ERROR\;
  }
  \If{$sgn = -1$}{
  	 \tcp{Fix shifts}
	 $\sigma \leftarrow -\sigma$\;
  }
  \Return{$\sigma$}\;
}
\end{algorithm}

\section{Conclusion}

In a work in preparation~\cite{Morain2026e}, we study optimal division
algorithms that can be deduced from~\cite{Cerri2007,Lezowski2014}.

\medskip
\noindent
{\bf Acknowledments.} The author wants to thank A.~Hermann of the
Information and Scientific editions (IES) at INRIA
Saclay for finding a lot of difficult-to-locate-and-scan articles,
including~\cite{Meissner1909}.

\def\noopsort#1{}\ifx\bibfrench\undefined\def\biling#1#2{#1}\else\def\biling#1#2{#2}\fi\def\Inpreparation{\biling{In
  preparation}{en
  pr{\'e}paration}}\def\Preprint{\biling{Preprint}{pr{\'e}version}}\def\Draft{\biling{Draft}{Manuscrit}}\def\Toappear{\biling{To
  appear}{\`A para\^\i tre}}\def\Inpress{\biling{In press}{Sous
  presse}}\def\Seealso{\biling{See also}{Voir
  {\'e}galement}}\def\Editor{\biling{Ed.}{R{\'e}d.}}

\end{document}